\documentclass{amsart}
\usepackage{amsxtra,amssymb,amscd,url,mathrsfs,tikz}
\usepackage[a4paper]{geometry}%
\usepackage{microtype}
\usepackage{enumitem}%
\usepackage{color} 
\usepackage[initials]{amsrefs}
\usepackage[colorlinks]{hyperref}

\DeclareMathOperator{\Adj}{Adj}

\def\prodp{\mathop{\prod\Bigl.^{*}}\limits}
\def\prodpt{\mathop{\prod\bigl.^{*}}\limits}

\newtheorem{theorem}{Theorem}[section] 
\newtheorem{lemma}[theorem]{Lemma}     

\newtheorem{proposition}[theorem]{Proposition}

\theoremstyle{definition}
\newtheorem{definition}{Definition}
\newtheorem{remark}{Remark}

\newcommand{\coloneqq}{\mathrel{\vcenter{\baselineskip0.5ex \lineskiplimit0pt
                     \hbox{\scriptsize.}\hbox{\scriptsize.}}}%
                     =}

\renewcommand{\le}{\leqslant}
\renewcommand{\ge}{\geqslant}
\renewcommand{\leq}{\leqslant}
\renewcommand{\geq}{\geqslant}
\newcommand{\bN}{\mathbf{N}}
\newcommand{\bZ}{\mathbf{Z}}

\newcommand{\C}{\mathscr{C}}

\newcommand{\abs}[1]{\left\lvert#1\right\rvert}

\newcommand{\bigo}[1]{O\mathopen{}\left(#1\right)}
\newcommand{\rest}{\mathclose{}|\mathopen{}}
\newcommand{\eps}{\varepsilon}
\newcommand{\Pe}{\text{\boldmath$P$}}
\newcommand{\Ee}{\text{\boldmath$E$}}
\newcommand{\intervalle}[4]{#1#2\mathpunct{}\nonscript\,,#3#4}
\newcommand{\icc}[2]{\intervalle{\left[}{#1}{#2}{\right]}}
\newcommand{\ioc}[2]{\intervalle{\left(}{#1}{#2}{\right]}}

\newcommand{\sst}[2]{\left\{#1\,:\,#2\right\}}

\title[Sieving in combinatorial structures]%
 {Expanders graphs and sieving in combinatorial structures} %
 \author{Florent Jouve}

\address{IMB, Universit\'e de Bordeaux, Talence, France}
\email{florent.jouve@math.u-bordeaux.fr}

\author{Jean-S\'ebastien Sereni}
\address{C.N.R.S., LORIA, Vand\oe uvre-l\`es-Nancy,
France.}
\email{sereni@kam.mff.cuni.cz}

\begin{document}

\maketitle

\begin{abstract}
        We prove a general large sieve statement in the context of 
        random walks on subgraphs of a given graph. This can be seen as a 
        generalization of previously known results where one performs 
        a random walk on a group enjoying a strong spectral gap property. 
        In such a context the point is to exhibit a strong uniform expansion
        property for a suitable family of Cayley graphs on quotients. In our 
        combinatorial approach, this is replaced by a result of Alon--Roichman
        about expanding properties of random Cayley graphs. Applying the
        general setting we show e.g.,~that with high probability (in a
        strong explicit sense) random coloured subsets of integers
        contain monochromatic (non-empty) subsets summing to~$0$,
        or that a random coloring of the edges of a complete graph 
        contains a monochromatic triangle.
\end{abstract}

\section*{Introduction}
The relevance of using families of expander graphs for studying objects or
solving problems coming from a broad variety of mathematical areas has been
emphasized in numerous ways in the recent years. Notably the combination of
sieving arguments together with expansion properties has proved particularly
efficient. Let us mention the groundbreaking work~\cite{BGS10} where the mix of
such techniques enabled the authors to detect almost primes in a variety of
non-Abelian situations.  A different kind of sieve together with the same
expansion properties have also been exploited in the context of group
theory~\cite{LM12}, or to obtain quantitative results in the probabilistic
Galois theory of arithmetic groups~\cite{JKZ13}.  In the sieving processes used
in the aforementioned works, one is naturally led to a crucial step where some
\emph{spectral gap} property is needed. Deep results about the groups involved
then come into play. Typically the required properties are provided by recent
breakthroughs in algebraic combinatorics that have led to strong forms of
Lubotzky's Property~$(\tau)$ (so-called \emph{superstrong approximation},
a culminating point of the study of which is the work by Golsefidy
and~Varj\'u~\cite{GV12}).  Indeed, the spectral gap result needed can be
tautolgically interpreted as a property of expansion of a certain family of
graphs.

The goal of the present paper is to establish a general large sieve inequality
in a purely combinatorial setting.  More precisely, we develop an axiomatic
version of sieve in the context of countable families of Cayley graphs, which are
randomly generated \emph{via} a random walk. This can be viewed as
a generalisation of the framework used e.g.~in~\cite[Chap.~7]{Kow08}
or~\cite{JKZ13} to a situation where spectral gap properties on groups are no
longer available. To circumvent the lack of algebraic structure we will exploit
the fact that random graphs are ``good expanders''. Precisely we will use a result
of Alon--Roichman~\cite{AlRo94} according to which a family of random Cayley
graphs is very likely to form a good family of expanders. However, further
difficulties arise: contrary to the usual situation, expansion properties are not
sufficient to ensure good enough cancellation in the correlation sums appearing.
To obtain the required cancellation we introduce structural properties and use
concentration arguments.  To the best of our knowledge, this point of view is new
within sieving contexts and it seems to us that it might lead to new uses
and theoretical study of sieves in a combinatorial setting.

After describing our general setting and proving our main result
(Theorem~\ref{thm:LS}), we briefly present some concrete uses of our result to
specific questions.

We describe several applications of our method to the study of typical
properties of subgraphs of a given graph. To produce random elements for which
we want to test if some given property holds, we perform a random walk on the
family of graphs studied, (cf. also~\cite{JKZ13}).  Another approach could
consist in quantifying the proportion of elements that satisfy an expected
property among a finite subset of the family of graphs considered.  For the
applications we have in mind this question would in fact be much easier. As
a matter of fact we do need to quantify proportions of ``good'' elements as
part of our sieving process.

The paper is organized as follows. In Section~\ref{Section:Setting} we state
and prove the main result and we emphasize the way in which Alon--Roichman's
theorem enables us to work in a setting which is combinatorial in nature
(whereas earlier works such as~\cite[Chap.~7]{Kow08} required a more
algebraic framework). The rest of the paper is devoted to applications of the
main result. Let us end the introduction by giving (rough) statements for some
of the concrete consequences of our main theoretical result
(Theorem~\ref{thm:LS}).  We conclude the paper with remarks on further
questions that may be of interest and that can be successfully investigated via
our method. We notably state a Ramsey type result (together with a sketch
of proof) obtained by suitably adapting the arguments used in the second
application.

\par\medskip

\textbf{Notation.} If~$X$ is a finite set, then $\# X$ and $\abs{X}$ synonymously
denote the cardinality of~$X$.

If $X$ is a finite graph, then $\Adj(X)$ is the adjacency operator sending
a $\mathbf{C}$-valued function on the vertices of~$X$ to the function
$(x\mapsto \sum_yf(y))$, where the sum is over the neighbors~$y$ of the
vertex~$x$. If $X$ is moreover $d$-regular (that is, every vertex of~$X$ has
degree~$d$), then the \emph{normalized adjacency operator} is
$\frac{1}{d}\cdot\Adj(X)$.

If~$G$ is a group and~$S\subset G$, then $X(G,S)$ is the Cayley graph on~$G$
with edge set~$S\cup S^{-1}\coloneqq\{s\in G\colon\text{$s\in S$ or $s^{-1}\in
S$}\}$.  If~$G$ is finite and Abelian, then $\hat{G}$ is the character group of
$G$.  If~$x$ is a non-negative real number, then $\lceil x\rceil$ and~$\lfloor
x\rfloor$ are the least integer greater than or equal to~$x$ and the greatest
integer smaller than or equal to~$x$, respectively. If $R$ is a positive
integer, then~$[R]$ is the set~$\{1,\dotsc,R\}$.  Given a probability space
$(\Omega,\Sigma,\Pe)$ and two events~$A$ and~$B$ such that $\Pe(B)\neq 0$, we
let~$\Pe(A\mid B)$ be the \emph{conditional probability} $\Pe(A\cap B)/\Pe(B)$.
If~$f$ and~$g$ are two real-valued functions defined on a set $\mathcal D$ and
depending on a set~$\mathcal P$ of parameters, then $f(x)\ll_{\mathcal P_0} g(x)$
means that there exists a positive constant~$C$ depending only on the
subset~$\mathcal P_0\subseteq \mathcal P$ such that $\abs{f(x)}\le C\abs{g(x)}$
whenever~$x\in\mathcal D$.

\section{The general setting and a large sieve for graphs}\label{Section:Setting}
\subsection{Random walk large sieve: statement of the main result}
Stating our main result requires some definitions and a precise description of
the general setting.  Let~$G$ be an Abelian group (in this section, the group
law is noted multiplicatively) and~$\Lambda\subset\bN$ be a (non necessarily
finite)  set of indices.  We suppose we are given a family
$(H_\ell)_{\ell\in\Lambda}$ of subgroups of~$G$ such that for each~$\ell$ the
index $n_\ell\coloneqq[G:H_\ell]$ is finite.  We let~$\rho_\ell\colon G\rightarrow
G/H_\ell$ be the canonical projection.
If $\ell$ and~$\ell'$ are two distinct elements of~$\Lambda$,
we define ${\rho_{\ell,\ell'}\colon G\rightarrow G/H_\ell\times G/H_{\ell'}}$
by ${g\mapsto (\rho_\ell(g),\rho_{\ell'}(g))}$.

We fix once and for all a probability space~$(\Omega,\Sigma,\Pe)$ and an
arbitrarily small real~${\delta\in (0,1)}$. We further set
\begin{equation}\label{eq:psi}
  \psi(\delta)\coloneqq 2\left((2-\delta)\log (2-\delta)+\delta\log\delta\right)^{-1}.
\end{equation}
For each $\ell \in \Lambda$, we define the quantity
\[
  \kappa(b_\ell,\ell;\delta)\coloneqq\left\lceil\psi(\delta)\bigl(\log n_\ell+b_\ell+\log 2\bigr)\right\rceil,
\]
where $b\coloneqq(b_\ell)$ is a parameter (a sequence of positive real
numbers). 

Now let~$s_1^{(\ell)},\dotsc,s_{\kappa(b_\ell,\ell;\delta)}^{(\ell)}$ be
independent identically distributed random variables taking values in~$G/H_\ell$.
The random walk on~$G$ we want to consider is obtained by lifting
the sets $\left\{s_1^{(\ell)},\dotsc,
s_{\kappa(b_\ell,\ell;\delta)}^{(\ell)}\right\}$ (and their ``inverses'' so
that all the graphs considered are then undirected) to~$G$.  To that purpose
we define the random variable
\[
S_\ell(b_\ell,\delta)\coloneqq\left\{s_1^{(\ell)},\dotsc,
s_{\kappa(b_\ell,\ell;\delta)}^{(\ell)}\right\}\cup\left\{(s_1^{(\ell)})^{-1},\dotsc,
(s_{\kappa(b_\ell,\ell;\delta)}^{(\ell)})^{-1}\right\},
\]
which takes values in the set of subsets of~$G/H_\ell$. For every index~$\ell\in\Lambda$ and
every integer~$m\in\{1,\dotsc,\kappa(b_\ell,\ell; \delta)\}$, we need to
choose a representative~$\tilde{s}_m^{(\ell)}\in G$ of~$s_m^{(\ell)}$. The
particular representative we choose is imposed by the following condition of
\emph{admissibility}. (As we shall establish later on, for each family of subgroups
there exists at most one admissible local sequence in the sense of Definition~\ref{def:adm}.)
\begin{definition}\label{def:adm}
Let~$(H_\ell)_{\ell\in\Lambda}$ be a fixed family of subgroups of finite index of~$G$ and, for each $\ell\in\Lambda$, let~$R_\ell$ be a set of representatives of~$G/H_\ell$. 
The sequence~$(H_\ell, R_\ell)_{\ell\in\Lambda}$ is called an \emph{admissible local sequence} for~$G$ if
\begin{enumerate}
\item[(i)] $\bigcap_{\ell\in\Lambda} H_\ell=\{1\}$; and
\item[(ii)] $\forall \ell,\ell'\in\Lambda,\quad \ell\neq\ell'\Rightarrow
R_\ell\subseteq H_{\ell'}.$
\end{enumerate}
\end{definition}

Let us assume that the sequence~$(H_\ell)$ of subgroups of~$G$
is such that there is an admissible local sequence~$(H_\ell,R_\ell)$ for~$G$. 
For each~$\ell$ and each~$m\in\{1,\dotsc,\kappa(b_\ell,\ell; \delta)\}$,
we choose the unique representative~$\tilde{s}_m^{(\ell)}$ of~$s_m^{(\ell)}$
in~$R_\ell$. 
The aforementioned uniqueness of an admissible local sequence ensures (see Lemma~\ref{lem:adm})
that this defines in a unique way elements~$\tilde{s}_m^{(\ell)}$ in~$G$.
Next we set
\[
\tilde{S}_\ell(b,\delta)\coloneqq\left\{\tilde{s}_1^{(\ell)},\dotsc,
\tilde{s}_{\kappa(b_\ell,\ell;\delta)}^{(\ell)}\right\}
\cup
\left\{(\tilde{s}_1^{(\ell)})^{-1},\dotsc,
(\tilde{s}_{\kappa(b_\ell,\ell;\delta)}^{(\ell)})^{-1}\right\}.
\]
The subset of~$G$ we use to perform a random walk on~$G$ is
\begin{equation}\label{def:S}
 S(b,\delta)\coloneqq \prodp_{\ell\in\Lambda}\left(\{1\}\cup  \tilde{S}_\ell(b_\ell,\delta)\right).
\end{equation}
Let us explain precisely what the notation means. If $A_1,\dotsc, A_k$ are $k$
subsets of~$G$, the product $\prod_{i=1}^k A_i$ is the subset $\{a_1\dotsc
a_k\colon a_i\in A_i\}$ of~$G$. Here the symbol $\prodpt$ means that for all
$\ell$ but finitely many of them the $\ell$-th factor picked equals~$1$.
Finally $S(b,\delta)$ is not seen as a random variable but as the product over
$\ell$ of elements either  equal to~$1$ or picked
in~$\tilde{S}_\ell(b_\ell,\delta)$ evaluated at a common $\omega\in\Omega$.  In
other words we fix once and for all an element $\omega$ of~$\Omega$; picking an
element of~$S(b,\delta)$ amounts to picking~$1$ or an element of some
$\tilde{S}_\ell(b_\ell,\delta)(\omega)$, and then computing the product of
these elements.

With notation as above, we perform the following random walk
on~$G$. It is defined the same way as in~\cite[Chap.~7]{Kow08}.
\[
\begin{cases}
    X_0=g_0 & \\
    X_{k+1}=X_k\xi_{k+1}&\quad\text{for~$k\ge0$},
   \end{cases}
\]
where $g_0$ is a fixed element in~$G$ and the steps~$\xi_k$ are independent,
identically distributed random variables with distribution
\[
\Pe(\xi_k=s)=\Pe(\xi_k=s^{-1})=p_s=p_{s^{-1}}
\]
for every~$k$ and every~$s\in S(b,\delta)$, and where~$(p_s)_s$ is a sequence
of positive real numbers indexed by~$S(b,\delta)$ such that
\[
\sum_{s\in  S(b,\delta)}p_s=1.
\]
Of course the random walk depends on the parameters~$b=(b_\ell)_\ell$
and~$\delta$. If $\Lambda$ is finite, the most natural such random walk is
certainly the one defined by uniformly distributing the steps, that is,
$p_s\coloneqq\# S(b,\delta)^{-1}$ for every~$s\in S(b,\delta)$.  In general, we
require that the sum of probabilities~$p_s$ over elements~$s\in S(b,\delta)$
that are mapped by~$\rho_{\ell,\ell'}$ to any given $(s',t')\in
S_\ell(b_\ell,\delta)\times S_{\ell'}(b_{\ell'},\delta)$ is not too small.
Precisely, we assume throughout the paper that for any given~$L\geq 1$, every
$\ell,\ell'\in\Lambda\cap\icc{1}{L}$ and every~$(s',t')\in
S_\ell(b_\ell,\delta)\times S_{\ell'}(b_{\ell'},\delta)$,
\begin{equation*}\label{star}
  \sum_{\substack{s\in S(b,\delta)\\ \rho_{\ell,\ell'}(s)=(s',t')}}p_s\geq \frac{1}{\kappa(b_L,L,\delta)}, \tag{$\star$}
\end{equation*}
which seems an intuitive generalisation of the
uniform distribution to a general set~$\Lambda$.

By studying the properties of the random walk~$(X_k)_k$ our aim is to describe
the behavior of a ``generic'' element~$g\in G$. To do so, we make use of
Kowalski's abstract large sieve procedure extensively described, together with
applications, in his book~\cite{Kow08}.  As in every sieve method, one can only
handle cases where the typical properties at issue can be detected locally.  To
be more precise, we fix for each~$\ell\in \Lambda$ a conjugacy invariant subset
$\Theta_\ell\subset G/H_\ell$. The probability we want to upper bound is
\[
      \Pe(\forall \ell\in \Lambda,\,\rho_\ell(X_k)\not\in\Theta_\ell).
\]

When applicable, the method shall produce effective upper bounds for the
probability with which $X_k$ satisfies a fixed property that can be detected by
the condition $\rho_\ell(X_k)\not\in \Theta_\ell$ for some~$\Theta_\ell\subset
G/H_\ell$. Our main result is the following abstract sieve statement.
We refer the reader to the book by Kowalski~\cite[Prop.~3.5]{Kow08} for a (self-contained) sieve statement
that Theorem~\ref{thm:LS} builds on.  For more information on the random walk
sieve used here, see also~\cite[Chap.~7]{Kow08}.
\begin{theorem}\label{thm:LS}
With notation as above, we set $\kappa_{N}\coloneqq \kappa(b_{N},N,\delta)$ for $N\in\mathbf{N}$ and
\[
C_0\coloneqq\sup_{L\in\bN_{>0}}\max_{\substack{\ell \neq\ell' \in \Lambda\\ L\leq\ell, \ell'\leq 2L}}\frac{\#
S_\ell(b_\ell,\delta)}{\#S_{\ell'}(b_{\ell'},\delta)}.
\]
Assume that condition~\eqref{star} holds.
Then there exists a positive real~$\nu$ such that for every positive integer~$k$, 
\begin{align*}
\Pe(\rho_\ell(X_k)\not\in\Theta_\ell,\,\forall \ell\in \Lambda_L)&\leqslant \Pe(C_0=\infty)+ \sum_{\ell\in\Lambda_L}e^{-b_\ell}\\
&+\left(1+\bigg(\sum_{\ell\in\Lambda_L}n_\ell\bigg)(1-\kappa_{2L}^{-2}\nu)^k\right)\left(\sum_{\ell\in\Lambda_L}\frac{\#\Theta_\ell}{n_\ell}\right)^{-1},
\end{align*}
where~$L$ is any fixed positive integer
$\Lambda_L\coloneqq\Lambda\cap\icc{L}{2L}$ and the constant~$\nu$ depends only on~$C_0$, the set~$S(b,\delta)$,
 and the distribution of the steps~$\xi_j$ (that is, the sequence~$(p_s)$).
\end{theorem}
In applications the probability that~$C_0$ is infinite will be very small
\emph{via} a suitable choice of parameters. Later on we prove a lemma
(Lemma~\ref{lem:proba}) the purpose of which is to bound
efficiently~$\Pe(C_0=\infty)$.  We end this section by proving the aforementioned uniqueness
of admissible sequences, when they exist.
\begin{lemma}\label{lem:adm}
Let~$(H_\ell)_{\ell\in\Lambda}$ be a family of subgroups of finite index
of~$G$. There exists at most one family~$(R_\ell)_{\ell\in\Lambda}$
where~$R_\ell$ is a set of representatives of~$G/H_\ell$ such that
$(H_\ell,R_\ell)$ is an admissible local sequence.
\end{lemma}
\begin{proof}
Let~$(H_\ell, R_\ell^{(1)})$ and~$(H_\ell, R_\ell^{(2)})$ be two admissible
local sequences for~$G$.  Fix~$\ell_0\in\Lambda$, and let~$r_1\in
R_{\ell_0}^{(1)}$ and~$r_2\in R_{\ell_0}^{(2)}$ be representatives of the same
element of~$G/H_{\ell_0}$. So there exists~$h\in H_{\ell_0}$ such that
$r_1=r_2h$.  For any~$\ell\neq \ell_0$, applying the reduction morphism
$\rho_\ell$ to the above equality yields that $\rho_\ell(h)=1$, because of
condition~(ii) of~Definition~\ref{def:adm}. Thus~$h\in H_{\ell'}$ for
every~$\ell'\in \Lambda$. Now, condition~(i) of~Definition~\ref{def:adm} implies
that~$h=1$, hence $r_1=r_2$. This shows that
$R_{\ell_0}^{(1)}=R_{\ell_0}^{(2)}$, thereby concluding the proof.
\end{proof}

\subsection{Cayley graphs on quotients and expansion}\label{cayleygraphs}
Let~$G$ be an Abelian\footnote{The assumption that $G$ is Abelian is
unnecessary for most of the results of this section, but our main result and
the applications we develop only involve Abelian groups, so we stick to this
case where the exposition is simpler.} group. 
We are interested in the properties of the Cayley graphs on the
groups~$(G/H_\ell)_{\ell\in\Lambda}$ with edges corresponding to the values taken by the random
variables~$s_i^{(\ell)}$ for~$i\in\{1,\dotsc,\kappa(b_\ell,\ell;\delta)\}$.
These graphs are regular: the regularity equals the number of distinct values
taken by the random variables~$s_i$.

Throughout the paper, if $\mathcal G$ is a $k$-regular graph, then the
\emph{eigenvalues of~$\mathcal G$} are the
eigenvalues of the normalized adjacency operator $k^{-1}\Adj(\mathcal G)$.
An eigenvalue $\lambda$ is \emph{non trivial} if $\abs{\lambda}\neq1$.
The \emph{spectral gap} $\varepsilon(\mathcal G)$ of~$\mathcal G$
is defined to be $\min\{1-\abs{\lambda}\colon \text{$\lambda$ is a non trivial eigenvalue
of~$\mathcal G$}\}$ (recall that the eigenvalue $-1$ occurs if and only if $\mathcal G$
is bipartite). We adopt the following definition for an expander graph, which
\textbf{slightly differs from the standard one}. In particular, for us,
a $k$-regular graph with spectral gap greater than~$1/2$ is a $\gamma$-expander
graph for any~$\gamma\in(0,1/2]$.
\begin{definition}\label{def-exp}
Let~$\gamma$ be a real number satisfying $0<\gamma\le1/2$. A $k$-regular graph $\mathcal G$ is a
$\gamma$-\emph{expander graph} if the spectral gap of~$\mathcal G$ is at least $\gamma$.
\end{definition}
\noindent

The reason for introducing the above setup is
a theorem of Alon \& Roichman~\cite[Th.~1]{AlRo94}, which has been
subsequently improved by Landau \& Russell~\cite[Th.~2]{LaRu04} and Loh \&
Schulman~\cite[Th.~1]{LoSc04}. The last improvement obtained so far, which
is the version we state and use, is due to Christofides \& Markstr\"om~\cite[Th.~5]{ChMa08}.
\begin{theorem}[Christofides--Markstr\"om]\label{AlonRoichman}
 With notation as above, fix an index $\ell$ in~$\Lambda$. For every
 $\delta\in(0,1/2]$, the probability
 that $X(G/H_\ell,\{s_1^{(\ell)},\dotsc,
 s_{\kappa(b_\ell,\ell;\delta)}^{(\ell)}\})$ is not a $\delta$-expander graph
 is less than $e^{-b_\ell}$.
\end{theorem}
The statement can be rephrased by saying it is highly probable that the Cayley
graph $X(G/H_\ell,\{s_1^{(\ell)},\dotsc,
s_{\kappa(b_\ell,\ell;\delta)}^{(\ell)}\})$ be a $\delta$-expander graph, the
counterpart being that the edge set has very large cardinality.  Note that the
definition of an expander graph we use is not completely equivalent to the
usual definition.  However, it is a standard fact that the (usual) expansion
property and the spectral gap property are closely related notions (see,
e.g.,~\cite[Th.~1.2.3]{DSV03}), which allows us to use our definition harmlessly
for our purposes.

Kowalski~\cite[Chap.~7]{Kow08} successfully combines large sieve techniques
with expansion properties in the setting of random walks on arithmetic groups.
We want to transpose this principle in a combinatorial setting. When adapting
Kowalski's work a non trivial issue comes from the fact that the expansion
property crucial to us is not automatically stable under Cartesian product.
Precisely, \emph{loc. cit.} relies on the fact that if~$S$ is a symmetric
generating system for~${\rm SL}_2(\bZ)$ and if $\pi_d\colon {\rm
SL}_2(\bZ)\rightarrow {\rm SL}_2(\bZ/d\bZ)$ is the reduction modulo~$d$ map for
some~$d\geq 2$, then the whole family of Cayley graphs on~${\rm SL}_2(\bZ)
/\ker \pi_d$ (with respect to the projection of~$S$) indexed by
\emph{squarefree integers} is expanding. In fact it would be enough to have the
same result with an index set replaced by the set of positive integers that are
squarefree and products of at most two primes. (However, for Kowalski's
purposes, considering the primes as the index set would not be sufficient.) To
obtain a suitable combinatorial analogue of this method, the forthcoming lemma
is sufficient.  It shows that expansion properties of Cayley graphs are
preserved, albeit only imperfectly, when one takes the Cartesian product of two
base groups.  The expansion ratio guaranteed by the lemma is strong enough for our purposes.
However we refer the interested reader to~\cite{ALW01} for a much more sophisticated
method that does produce expander ``product Cayley graphs''.  (Recall that the
definition of ``expander graph'' we use slightly differs from the standard
one.)
\begin{lemma}\label{lem:NiceImage}
Let~$\delta\in(0,1/2]$.  With notation as above assume that~$X(G,S)$
and~$X(H,T)$ are $\delta$-expander Cayley graphs on finite Abelian groups~$G$
and~$H$ (with edge set defined by~$S\subseteq G$ and~$T\subseteq H$, respectively).
Then for every~$(x_0,y_0)\in G\times H$ with~$x_0^2=1=y_0^2$, the Cayley graph
${X(G\times H,(S\times\{y_0\})\cup(\{x_0\}\times T))}$ is
a~$((1+\gamma)^{-1}\delta)$-expander graph, where
\[
      \gamma\coloneqq
          \max\left\{\frac{\abs{S\cup S^{-1}}}{\abs{T\cup T^{-1}}},\frac{\abs{T\cup
T^{-1}}}{\abs{S\cup S^{-1}}}\right\}.\]
\end{lemma}
\begin{proof}
For convenience, set $Y\coloneqq(S\times\{y_0\})\cup(\{x_0\}\times T)$,
$S^*\coloneqq S\cup S^{-1}$, ${T^*\coloneqq T\cup T^{-1}}$ and ${Y^*\coloneqq
Y\cup Y^{-1}}$.  The eigenfunctions of the normalized adjacency operator
on~$X(G\times H,Y)$ are of the form
\[
(\chi,\tau)\colon (g,h)\mapsto \chi(g)\tau(h),
\]
for characters~$\chi\in\hat{G}$ and~$\tau\in\hat{H}$. The corresponding
eigenvalues are of the form
\[
\lambda_{\chi,\tau}\coloneqq\frac{1}{\abs{S^*}+\abs{T^*}}\sum_{(g,h)\in Y^*}\chi(g)\tau(h).
\]
Since~$x_0^2=1=y_0^2$ the sum splits as follows:
\begin{equation}
 \left(\abs{S^*}+\abs{T^*}\right)\lambda_{\chi,\tau}=\tau(y_0)\sum_{g\in
S^*}\chi(g)+\chi(x_0)\sum_{h\in T^*}\tau(h).
      \label{eq:lem6}
\end{equation}
We deduce that
\[
\abs{\lambda_{\chi,\tau}}\leq
\frac{\abs{S^*}}{\abs{S^*}+\abs{T^*}}\abs{\frac{1}{\abs{S^*}}
\sum_{g\in S^*}\chi(g)}+\frac{\abs{T^*}}{\abs{S^*}+\abs{T^*}}\abs{\frac{1}{\abs{T^*}}\sum_{h\in
T^*}\tau(h)}.
\]
If both~$\chi$ and~$\tau$ are non-trivial, then $\abs{\lambda_{\chi,\tau}}\leq
1-\delta$ since each of~$X(G,S)$ and~$X(H,T)$ are $\delta$-expanders where
$\delta\in(0,1/2]$.  If $\chi$ is trivial and~$\tau$ is non trivial, we obtain
instead
\[
\abs{\lambda_{\chi,\tau}}\leq 1-\delta\left(1+\abs{S^*}/\abs{T^*}\right)^{-1},
\]
hence the result by symmetry of the roles played by~$G$ and~$H$.
\end{proof}

To better comprehend Lemma~\ref{lem:NiceImage}, we give several examples, which
also allow us to demonstrate the necessity of its hypothesis and the optimality
of the bound given.  For a positive integer~$n$,
we let~$\mathbf{Z}_n$ be the cyclic group of order~$n$.
Consider first the case where both~$G$ and~$H$ are
$\mathbf{Z}_4$, with~$S$ and~$T$ each consisting of a generating element
of~$\mathbf{Z}_4$.  Thus the graphs~$X(G,S)$ and~$X(G,T)$ are isomorphic to the
undirected cycle~$C_4$ with~$4$ vertices. (Recall that $X(G,S)=X(G,S^*)$ by the
definition.) The spectral gap of~$C_4$ is~$1$. Now choose~$x_0$ and~$y_0$ to be
the neutral elements of~$G$ and~$H$, respectively. The hypothesis of
Lemma~\ref{lem:NiceImage} are thus satisfied. Note that~$\gamma=1$, so according
to this lemma, the graph~$X\coloneqq X(G\times H,(\{x_0\}\times T)\cup(S\times \{y_0\})$ should have
spectral gap at least $(1+1)^{-1}\cdot1=\frac12$. To check this, observe
that~$X$ is the $4$-regular graph depicted in Figure~\ref{fig:lem6}: it consists
of two disjoint cycles of size~$8$ the vertices of which are ``linked using
cycles of length~$4$''. This graph indeed has spectral gap exactly~$\frac12$. This can
actually be directly deduced from the proof of Lemma~\ref{lem:NiceImage} by
using~\eqref{eq:lem6}, thereby obtaining
a precise expression for the eigenvalues of the product graph.  More generally,
one deduces that performing the same construction as we just did but starting
from~$\mathbf{Z}_{2k}$ for any integer~$k\ge2$ yields an infinite family of
examples where the bound given by Lemma~\ref{lem:NiceImage} is attained,
showing its optimality. (The spectral gap of the two (isomorphic) starting
graphs will be~$1-\abs{\cos(\pi(k+1)/k)}$ and that of the product graph exactly
half this quantity.)

The hypothesis that $x_0$ and~$y_0$ must be elements with order at most~$2$ in their respective groups
is necessary, as is seen by taking $G\coloneqq\mathbf{Z}_3\times\mathbf{Z}_5=\langle\sigma\rangle\times\langle\tau\rangle$
and~$H$ the dihedral group of order~$6$. Letting~$\mu$ be a generator of~$H$, we set
$S\coloneqq\{\sigma\}$, $T\coloneqq\{\mu\}$, $x_0\coloneqq\sigma\tau$ and~$y_0\coloneqq\pi^3$.
(Thus $y_0$ is of order~$2$ while~$x_0$ is of order~$15$.)
The graph~$X(G,S)$ consists of five disjoint triangles while~$X(H,T)$ consists of two disjoint
cycles of length~$6$. Consequently each of these graphs has spectral gap~$\frac12$. However,
the spectral gap of the graph $X(G\times H,((\{x_0\}\times T)\cup(S\times\{y_0\})))$ is
less than~$0.045$, which is less than~$\frac12\cdot(1+\gamma)^{-1}=\frac14$.

\begin{figure}[!ht]
\begin{center}
\begin{tikzpicture}[vertex/.style={circle, draw=black, fill=black, inner sep=0.5pt, minimum
      size=6pt},edge/.style={thick}]
\draw (0,2) node[vertex] (t) {};
\draw (0,1) node[vertex] (tm) {};
\draw (.5,.72) node[vertex] (tr) {};
\draw (-.5,.72) node[vertex] (tl) {};
\draw (1,0) node[vertex] (rm) {};
\draw (2,0) node[vertex] (r) {};
\draw (-1,0) node[vertex] (lm) {};
\draw (-2,0) node[vertex] (l) {};

\draw (0,-2) node[vertex] (b) {};
\draw (0,-1) node[vertex] (bm) {};
\draw (.5,-.72) node[vertex] (br) {};
\draw (-.5,-.72) node[vertex] (bl) {};

\draw (-1.5,1.5) node[vertex] (dtl) {};
\draw (1.5,1.5) node[vertex] (dtr) {};
\draw (-1.5,-1.5) node[vertex] (dbl) {};
\draw (1.5,-1.5) node[vertex] (dbr) {};

\draw[edge] (t)--(dtr)--(tm)--(dtl)--(t);
\draw[edge] (b)--(dbr)--(bm)--(dbl)--(b);
\draw[edge] (t)--(tr)--(bm)--(tl)--(t);
\draw[edge] (b)--(br)--(tm)--(bl)--(b);
\draw[edge] (dtl)--(l)--(dbl)--(lm)--(dtl);
\draw[edge] (dtr)--(r)--(dbr)--(rm)--(dtr);
\draw[edge] (tm)--(br)--(lm)--(tr)--(bm)--(tl)--(rm)--(bl)--(tm);
\draw[edge] (tl)--(l)--(bl);
\draw[edge] (tr)--(r)--(br);
      \end{tikzpicture}
      \caption{The product graph $X(\mathbf{Z}_4\times\mathbf{Z}_4,\{(1,\sigma),(\sigma,1)\})$, where
      $\sigma$ is a generating element of~$\mathbf{Z}_4$.}\label{fig:lem6}
 \end{center}
\end{figure}
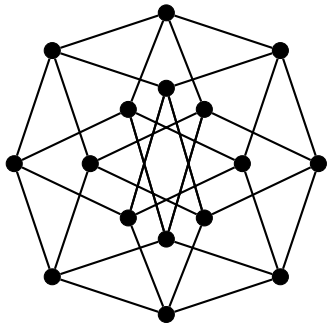

In applications, it is important to keep control of the ``spectral gap loss'',
that is, the size of the parameter~$\gamma$ appearing in
Lemma~\ref{lem:NiceImage}. To do so we need the following technical lemma,
which asserts in a precise quantitative way that it is harmless to suppose that
the random variables~$s_i^{(\ell)}$ take distinct values when evaluated
simultaneously and as long as~$n_\ell$ is fairly larger
than~$\kappa(b_\ell,\ell;\delta)$. This is a simple application of standard
concentration principles.
\begin{lemma}\label{lem:proba}
Keeping notation as above, fix $\ell\in\Lambda$. For readability, $\kappa(b_\ell,\ell;\delta)$ is abbreviated
to $\kappa_\ell$. Let~$X$ be the random variable that counts the number
of distinct values in the multi-set~$\{s_1^{(\ell)},\dotsc,s_{\kappa_\ell}^{(\ell)}\}$.
One has:
\begin{enumerate}[label=\upshape(\alph*),ref={(\alph*)}]
      \item if there exists a positive real $\varepsilon$, \emph{independent
            of~$\ell$}, such that $1-\kappa_\ell/n_\ell>\varepsilon$, then
\[
      \Pe(X<\kappa_\ell/2)\le2\exp(-\kappa_\ell\varepsilon^2/8);
\]\label{it:lemprobai}

\item if $n_\ell\le\kappa_\ell$, then
\[
      \Pe(X<n_\ell/2)\le\binom{n_\ell}{\lceil n_\ell/2\rceil}\cdot2^{-\kappa_\ell}.
\]\label{it:lemprobaii}
\end{enumerate}
\end{lemma}
\begin{proof}
Let~$x_1,\dotsc,x_{n_{\ell}}$ be the elements of~$G/H_\ell$.
For each~$i\in[n_\ell]$, let~$X_i$ be the $0$-$1$ random variable that
is equal to $1$ if $x_i\in\sst{s_j^{\ell}}{1\le j\le n_\ell}$. Notice that
$X=\sum_{i=1}^{n_\ell}X_i$. Consequently, the linearity of expectation implies that
$\Ee(X)=\sum_{i=1}^{n_\ell}\Ee(X_i)$. Moreover, for each~$i\in[n_{\ell}]$,
\[
   \Ee(X_i)=1-\Pe(X_i=0)=1-\left(1-\frac{1}{n_\ell}\right)^{\kappa_\ell}\ge1-\exp(-\kappa_\ell/n_\ell)\ge
   \kappa_\ell/n_\ell-1/2\cdot(\kappa_\ell/n_\ell)^2
\]
so that $\Ee(X)\ge\kappa_\ell-1/2\cdot\kappa_\ell^2/n_\ell$.

Now, since~$X$ is determined by $\kappa_\ell$ independent trials and, for every
possible outcome of the trials changing the outcome of any one trial can
affect~$X$ by at most~$1$, the Simple Concentration Bound~\cite[p.~79]{MR02}
yields that for every positive number~$t$,
\[
   \Pe\left(\abs{X-\Ee(X)}>t\right)\le2\exp\left(-\frac{t^2}{2\kappa_\ell}\right).
\]
Therefore, to prove~\eqref{it:lemprobai} one sets~$t\coloneqq
\kappa_\ell/2\cdot(1-\kappa_\ell/n_\ell)$. This implies that
\[
   \Pe\left(X<\kappa_\ell/2\right)\le2\exp\left(-\frac{\kappa_\ell}{8}(1-\kappa_\ell/n_\ell)^2\right)\le2\exp\left(-\frac{\kappa_\ell\varepsilon^2}{8}\right).
\]

To prove~\eqref{it:lemprobaii} one rather proceeds as follows. Notice that
$X<n_\ell/2$ if and only if there exists a subset~$H'$ of~$G/H_\ell$ of
size~$\lceil n_\ell/2\rceil$ such that $\sst{s_i^\ell}{1\le
i\le\kappa_\ell}\cap H'=\varnothing$. This happens with probability at
most~$2^{-\kappa_\ell}$.  Consequently, we infer that
\[
   \Pe(X<n_\ell)\le\binom{n_\ell}{\lceil n_\ell/2\rceil}\cdot2^{-\kappa_\ell}.
\]
\end{proof}

\subsection{Random walk large sieve: proof of the main result}\label{sub:rw}
We first state an easy consequence of the definition of~$S(b,\delta)$ that is useful in our sieving procedure. 
\begin{lemma}\label{lem:lindisjoint}
For all distinct integers $\ell,\ell'\in\Lambda$, one has
\[
      \rho_\ell\left(S(b;\delta)\right)=S_\ell(b_\ell,\delta)\cup\{1\}\quad\text{and}\quad\rho_{\ell,\ell'}\left(S(b,\delta)\right)=\left(S_\ell(b_\ell,\delta)\cup\{1\}\right)\times \left(S_{\ell'}(b_{\ell'},\delta)\cup\{1\}\right).
\]
\end{lemma}
\begin{proof}
We use condition~(ii) in Definition~\ref{def:adm}: the image by~$\rho_\ell$ of~$S(b,\delta)$ is the product of elements 
all equal to~$1$ except maybe for the $\ell$-factor which can be any element of~$\rho_\ell(\tilde{S}_\ell(b_\ell,\delta)\cup\{1\})$,
that is, any 
element of~$S_\ell(b_\ell,\delta)\cup\{1\}$.

The second equality is obtained using the same argument.
\end{proof}

We now define one last piece of useful notation before starting the proof of
Theorem~\ref{thm:LS}.  For indices~$\ell$ and~$\ell'$ in~$\Lambda$, we set
$G_{\ell,\ell'}\coloneqq G/H_\ell\times G/H_\ell'$ if $\ell\neq \ell'$ and
$G_{\ell}(=G_{\ell,\ell'})\coloneqq G/H_\ell$ otherwise.  The proof
of~Theorem~\ref{thm:LS} is based on an adaptation of that
of~\cite[Prop.~7.2]{Kow08}.
\begin{proof}[Proof of Theorem~\ref{thm:LS}]
Fix a real number~$\delta$ in~$\ioc{0}{1/2}$ and
let us split the probability we are interested in:
\begin{align}\label{DebutMajProb}
&\Pe\left(\forall\ell\in\Lambda_{L},\, \rho_\ell(X_k)\not\in\Theta_\ell\right)\leqslant \Pe(C_0=\infty) \\
&+\Pe\left(\exists \ell\in\Lambda_{L},\,\text{$X(G/H_\ell,\rho_\ell(S(b,\delta)))$ is not a $\delta$-expander}\right) \nonumber\\ 
 &+\Pe\left((C_0<\infty)\wedge (\forall\ell\in\Lambda_{L},\,
\text{$X(G/H_\ell,\rho_\ell(S(b,\delta))$ is a $\delta$-expander and }
\rho_\ell(X_k)\not\in\Theta_\ell)\right).\nonumber
\end{align}

As we shall see, the third summand can be bounded from above using sieving techniques.
Moreover, the second summand can be handled by invoking Theorem~\ref{AlonRoichman}. Indeed,
since
$\rho_\ell(S(b,\delta))=S_\ell(b_\ell,\delta)\cup\{1\}$ by Lemma~\ref{lem:lindisjoint}, we can show that the following statement holds (note that the
statement would be trivial if we were only interested in edge-expansion):
\begin{align*}
\Pe(&\exists\ell\in\Lambda_{L},\,\text{$X(G/H_\ell,\rho_\ell(S(b,\delta)))$
is not a $\delta$-expander})\\
&\le\Pe(\exists
\ell\in\Lambda_{L},\,\text{$X(G/H_\ell,S_\ell(b_\ell,\delta))$ is not a
$\delta$-expander}).
\end{align*}

\noindent
This inequality is a consequence of Lemma~\ref{as:1} that we state and prove at
the end of this section. Applying Theorem~\ref{AlonRoichman} yields that
\[
 \Pe(\exists
\ell\in\Lambda_{L},\,\text{$X(G/H_\ell,S_\ell(b_\ell,\delta))$ is not a
$\delta$-expander})\le\sum_{\ell\in \Lambda_{L}}e^{-b_\ell}.
\]

Let us now turn to the third summand of the right side of~\eqref{DebutMajProb}.
First, notice that
\begin{align*}
&\Pe\left((C_0<\infty)\wedge(\forall \ell\in\Lambda_{L},\,\text{$X(G/H_\ell,\rho_\ell(S(b,\delta)))$ is a $\delta$-expander and $\rho_\ell(X_k)\not\in\Theta_\ell$})\right)\\
\leq & \Pe(\forall \ell\in\Lambda_{L},\,
\rho_\ell(X_k)\not\in\Theta_\ell\mid C_0<\infty\text{ and }\forall \ell\in\Lambda_{L},\,
\text{$X(G/H_\ell,\rho_\ell(S(b,\delta)))$ is a $\delta$-expander}).
\end{align*}

Now we are in a situation close to the axiomatic sieve method developed in~\cite{Kow08}. 
We fix (non-necessarily distinct) indices~$\ell$ and~$\ell'$ in~$\Lambda_{L}$ (defining~$\rho_{\ell,\ell}$ to
be~$\rho_\ell$) and a character~$\lambda$ of~$G_{\ell,\ell'}$
\[
\lambda\colon G\stackrel{\rho_{\ell,\ell'}}{\rightarrow} G_{\ell,\ell'} \stackrel{\lambda_0}{\rightarrow}\mathbf{C}^\times,
\]
factoring through~$G_{\ell,\ell'}$ in such a way that $\lambda_0$ is a non trivial character of~$G_{\ell,\ell'}$.

\par
\medskip
We first prove the following statement.  Assume that $C_0<\infty$ and
$X(G/H_\ell,\rho_\ell(S(b,\delta)))$ is a $\delta$-expander for
every~$\ell\in\Lambda$. We assert that there exists a positive constant~$\nu$
depending only on~$C_0$, the set~$S(b,\delta)$ and the distribution of the
steps~$\xi_j$ (that is, the sequence~$(p_s)$), such that
\begin{equation}\label{eq:Eupperbound}
\abs{\Ee(\lambda(X_k)) }\leqslant (1-\kappa_{2L}^{-2}\nu)^k.
\end{equation}
Consider:
\[
      M\coloneqq\Ee(\lambda(\xi_k))=\sum_{s\in S(b,\delta)}{p(s)\lambda(s)},
\]
which is a well-defined element of~$\mathbf{C}^\times$ since the series defining~$M$ converges absolutely. Let us also consider the
complex number~$M^+\coloneqq1-M$.

Note that~$M$ and~$M^+$ are in fact real numbers since the set~$S(b,\delta)$
as well as the distribution of the steps~$\xi_k$ are symmetric. We also need to define
\[
      N_0\coloneqq\Ee(\lambda(X_0))=\sum_{t\in T}{\Pe(X_0=t)\lambda(t)}\in {\bf C}^\times,
\]
where~$T$ is a fixed (finite) subset of~$G$ containing the starting point~$g_0$ of the random walk~$(X_k)$.
(For simplicity one can assume that $T=\{g_0\}$.)

The random variables~$X_0$ and~$\xi_k$ being independent, it follows that for every positive integer~$k$,
\[
      \Ee(\lambda(X_k))=N_0 M^k.
\]
We have $\abs{N_0}\leq 1$ and we compute
  \begin{align*}
M^+&=\sum_{s\in
  S(b,\delta)}{p_s(1-\lambda(s))}=\sum_{(s',t')\in S_\ell(b_\ell,\delta)\times S_{\ell'}(b_{\ell'},\delta))}\left(\sum_{\substack{s\in S(b,\delta)\\ 
        \rho_{\ell,\ell'}(s)=(s',t')}}p_s\right)(1-\lambda_0(s',t'))\\
&\geqslant\frac{1}{\kappa(b_{2L},2L,\delta)^2}\min_{\psi\neq 1}\max_{(s',t')\in
S_\ell(b_\ell,\delta)\times S_{\ell'}(b_{\ell'},\delta))}(1-\psi(s',t'))\,
  \end{align*}
  where~$\psi$ runs over the non trivial characters of~$G_{\ell,\ell'}$. With the same notation we deduce that for any~$a\in S_\ell(b_\ell,\delta)$ and any~$b\in S_{\ell'}(b_{\ell'},\delta)$ both of order at most~$2$,
\[
      M^+\geq \kappa(b_{2L},2L,\delta)^{-2}\min_{\psi\neq 1}\max_{(s',t')\in
            S_\ell(b_\ell,\delta)\times\{b\}\cup
            \{a\}\times S_{\ell'}(b_{\ell'},\delta))}(1-\psi(s',t')).
\]
Lemma~\ref{lem:NiceImage} asserts that the family of Cayley graphs with vertex
set~$G_{\ell,\ell'}$ and edge set~$S_\ell(b_\ell,\delta)\times\{b\}\cup
\{a\}\times S_{\ell'}(b_{\ell'},\delta))$ is a family
of~$(1+C_0)^{-1}\delta$-expanders as soon as the family
$(X(G_\ell,S_\ell(b_\ell,\delta)))_{\ell\in\Lambda}$ is a family
of~$\delta$-expanders. Thus we can appeal to the translation of Lubotzky's
property~$(\tau)$ into the property of expansion of the corresponding Cayley
graphs (see~\cite[Prop.~2.5]{LZ}) to justify the existence of a positive
constant~$\nu(C_0,S(b,\delta),(p_s))$ which is uniform in~$\ell,\ell'\in\Lambda$
and such that 
\[
      M^+\geq \kappa(b_{2L},2L,\delta)^{-2}\nu(C_0,S(b,\delta),(p_s)).
\]
To conclude the proof of the claim it suffices to observe that, because of our
definition of expansion, the fact that the
family~$X(G_\ell,\rho_\ell(S(b,\delta)))$ is assumed to be a family
of~$\delta$-expanders implies that these Cayley graphs are not bipartite and
hence producing a lower bound for~$1+M$ is not required.
  
\par\medskip
We can now finish the proof by using Kowalski's large sieve inequality~\cite[Prop.~3.7]{Kow08}. We obtain
\begin{align*}
 & \Pe(\forall \ell\in\Lambda_{L},\,  \rho_\ell(X_k)\not\in\Theta_\ell\mid C_0<\infty\text{ and } \forall \ell\in\Lambda_{L}\text{ $X(G/H_\ell,\rho_\ell(S(b,\delta)))$ is a $\delta$-expander})\\
\leqslant &\Delta(X_k; L)\left(\sum_{\ell\in\Lambda_L}\frac{\#\Theta_\ell}{n_\ell}\right)^{-1},
\end{align*}
where one has the theoretical upper bound:
\[
 \Delta(X_k; L)\leqslant \max_{\ell\in\Lambda_{L}}\max_{\chi\in
 \mathcal B_\ell^*}\sum_{\ell'\in\Lambda_{L}}\sum_{\chi'\in \mathcal
 B_{\ell'}^*} \abs{W(\chi,\chi')},
\]
with
\[
 W(\chi,\chi')\coloneqq\Ee\left(\chi\rho_\ell(X_k)\overline{\chi'\rho_{\ell'}(X_k)}\right)=\Ee\left([\chi,\overline{\chi'}]\rho_{\ell,\ell'}(X_k)\right).
\]
Here for any~$\ell\in \Lambda$ we let~$\mathcal B_\ell$ be the
character group of~$G/H_\ell$. We set further $\mathcal B_\ell^*\coloneqq\mathcal
B_\ell\setminus\{1\}$. Finally if $\psi_i$ is a character of a finite Abelian group~$G_i$ for~$i\in\{1,2\}$,
then we let~$[\psi_1,\psi_2]$ be the character~$\pi\otimes\tau$ of~$G\times H$ if $G\neq H$ or of~$G$ otherwise.

In our setting, using~\cite[Lemma~3.4]{Kow08} we deduce that
\[
      [\chi,\overline{\chi'}]\rho_{\ell,\ell'}=\delta((\ell,\pi),(\ell',\tau)){\bf 1}+[\chi\overline{\chi'}]_0\rho_{\ell,\ell'},
\]
where~$\delta(\cdot,\cdot)$ is the Kronecker symbol and~$[\chi,\overline{\chi'}]_0$ is the component of~$[\chi,\overline{\chi'}]$ orthogonal to the trivial character~$\mathbf{1}$. Thus applying~\eqref{eq:Eupperbound} to 
\[
      \lambda\coloneqq [\chi,\overline{\chi'}]_0\rho_{\ell,\ell'}
\]
we obtain
\[
    \abs{\Ee([\chi,\overline{\chi'}]_0\rho_{\ell,\ell'}(X_k))}\le (1-\nu\kappa_{2L}^{-2})^k.
\]
Putting everything together we deduce as wished that
  \begin{align*}
 & \Pe(\forall \ell\in\Lambda_{L},\,
 \rho_\ell(X_k)\not\in\Theta_\ell\mid C_0<\infty\text{ and }\forall \ell\in\Lambda_{L}
  \text{ $X(G/H_\ell,\rho_\ell(S(b,\delta)))$ is a $\delta$-expander})\\
\leqslant
&\left(1+\bigg(\sum_{\ell\in\Lambda_L}n_\ell\bigg)(1-\nu\kappa_{2L}^{-2})^k\right)\left(\sum_{\ell\in\Lambda_L}\frac{\#\Theta_\ell}{n_\ell}\right)^{-1}.
 \end{align*}
\end{proof}

It remains to prove the following statement.
\begin{lemma}\label{as:1}
Let~$G_0$ be an Abelian group and let~$S$ be a subset of~$G_0$.
If $X(G_0,S)$ is a
$\delta$-expander graph, then so is~$X(G_0,S\cup\{1\})$.
\end{lemma}
\begin{proof}
The statement is trivially true if $1\in S$, so we assume that $1\not\in
S$. Set~$S^*\coloneqq S\cup S^{-1}$ and~$s^*\coloneqq\#S^*$.
Recall that Definition~\ref{def-exp} implies that
$\delta\in (0,1/2]$.
To prove the statement, it suffices to show that
every non trivial eigenvalue $\lambda'$ of~$X(G_0,S\cup\{1\})$ is such that
$\abs{\lambda'}\le 1/2$ or $\abs{\lambda'}\le\abs{\lambda}$ for some
non trivial eigenvalue $\lambda$ of~$X(G_0,S)$.

Let~$\lambda'$ be a non trivial eigenvalue of~$X(G_0,S\cup\{1\})$.
Using the usual convention according to which a loop contributes~$2$ to the
degree of a vertex, we deduce that
$\lambda'=(2+s^*)^{-1}(\sum_{s\in S^*}\chi(s)+\chi(1))$
for some non trivial character $\chi$ of~$G_0$. Therefore,
\[
\lambda'
=\frac{s^*}{2+s^*}\lambda+\frac{1}{2+s^*}=\lambda+\frac{1-2\lambda}{2+s^*},
\]
where $\lambda\coloneqq(s^*)^{-1}\sum_{s\in S^*}\chi(s)$ is a non trivial eigenvalue
of~$X(G_0,S)$.

Consequently, it is enough to prove that if $\abs{\lambda'}>1/2$, then
$\abs{\lambda'}\le\abs{\lambda}$. Suppose, on the contrary, that
$\abs{\lambda'}>1/2$ and $\abs{\lambda'}>\abs{\lambda}$. Then $\lambda'>0$. Indeed, otherwise
$\lambda\le-1/s^*<0$ and hence
$-\lambda+\frac{2\lambda-1}{2+s^*}>\abs{\lambda}=-\lambda$ implies that $\lambda>1/2$,
a contradiction.

Hence, $\lambda'>1/2$, which yields that $\lambda>1/2$. However, this implies
that $\frac{1-2\lambda}{2+s^*}<0$, so
that $\lambda>\lambda'=\abs{\lambda'}$, contrary to our assumption.
This finishes the proof.
\end{proof}

\section{Illustrative examples}\label{sec:app}
This section illustrate how Theorem~\ref{thm:LS} can be applied to various
classical topics. Let us state two bounds on~$\psi$ that are useful in the forthcoming applications.
An elementary study of the function~$\psi$, which extends continuously to~$\icc{0}{1/2}$, shows that $\psi$ is
increasing on that interval and therefore: 
\begin{equation}\label{eq:boundpsi}
 1.442\dotso\simeq (1/\log 2)\le\psi(\delta)\le 4/\log (27/16)\simeq 7.644\dotso
\end{equation}
for any~$\delta\in (0,1/2]$.
We now present some applications inspired by the classical Ramsey Theory.

\subsection{Towards a quantitative infinite Ramsey theory}\label{sub:InfiniteRamsey}
Our purpose is to illustrate how our method can
be applied in the context of infinite Ramsey theory. Let us first recall the
result we have in mind, established by Ramsey~\cite{Ram30}.  Given a set~$X$
and a non-negative integer~$r$, we define~$X^{(r)}$ to be the collection of
all subsets of~$X$ of cardinality~$r$.

\begin{theorem}[Infinite Ramsey Theorem~\cite{Ram30}]\label{th:InfinteRamsey}
Let~$X$ be some countably infinite set. Let~$c$ and~$r$ be positive
integers. Consider a given colouring~$f\colon X^{(r)}\rightarrow \bZ/c\bZ$ of
the elements of~$X^{(r)}$ in~$c$ different colours. Then there exists some
infinite subset~$A$ of~$X$ such that the function~$f$ is constant on~$A^{(r)}$,
that is, all subsets of~$A$ of cardinality~$r$ have the same image under~$f$.
\end{theorem}

For every function~$f$, the \emph{support} of~$f$ is the set of all elements~$e$ in the
domain of~$f$ such that $f(e)\neq0$.
As in the statement of Ramsey's Theorem, fix positive integers~$c$ and~$r$. As
our base set we choose~$X\coloneqq \bN_{\ge1}$. The set~$\C^{(r)}$ of all possible
$c$-colourings of subsets of cardinality~$r$ of~$X$ may be endowed with a group
structure inherited from that of~$\bZ/c\bZ$.
Explicitly, the addition of two elements~$f$ and~$g$ of~$\C^{(r)}$ is
formally defined by
\[
      f+g\colon X^{(r)}\rightarrow {\bZ/c\bZ},\qquad A\mapsto f(A)+g(A).
\]
The neutral element is the function that is identically~$0$.

We also fix an auxiliary positive integer~$j$ and we set~$\Lambda\coloneqq\bN_{\ge1}$.
We consider the subsets~${I_\ell^{(r,j)}\coloneqq \{(r+j)(\ell-1)+1,\dotsc,
(r+j)\ell\}}$ of~$X$ indexed by~$\ell\in \Lambda$.
If $j$ and~$r$ are fixed, then $I_{\ell}^{(r,j)}$ is an
integral interval of size~$r+j$ and different indices~$\ell$ and~$\ell'$ give
rise to disjoint intervals~$I_\ell^{(r,j)}$ and~$I_{\ell'}^{(r,j)}$.
For~$\ell\in\Lambda$, let~$E_\ell^{(r)}$
be the set of subsets of size~$r$ of~$I_\ell^{(r,j)}$. Let~$C_\ell$ be the
collection of all colourings supported on~$E_\ell^{(r)}$ and let~$H_\ell$ be the
subgroup of all colourings of~$\C^{(r)}$ supported on the complement of~$E_\ell^{(r)}$ in~$X^{(r)}$.
This way $C_\ell$ is a set of representatives for the
quotient~$\C^{(r)}/H_\ell$. Indeed,
no two distinct functions in~$C_\ell$ are congruent modulo an
element of~$H_\ell$.  Moreover, for any~$f\in\C$, let~$f_C$ be the coloring
equal to~$f$ on~$E_\ell^{(r)}$ and equal to~$0$ everywhere else.
It follows that $f_C\in C_\ell$ and $f-f_C\in H_\ell$, or equivalently $f\equiv f_C (\bmod
H_\ell)$. 
Let~$\rho_\ell\colon\C^{(r)}\rightarrow \C^{(r)}/H_\ell$ be the canonical surjection.
The disjointness of the sets~$I_\ell^{(r,j)}$ readily implies that
$(H_\ell,C_\ell)$ is an admissible local sequence for~$\C^{(r)}$. In addition, we note
that $\abs{E_\ell^{(r)}}=\binom{r+j}{r}$. Summing-up, we thus established the following statement.
\begin{lemma}\label{lem:coloringInfiniteRamsey}
The sequence~$(H_\ell,C_\ell)$ is an admissible local sequence for~$\C^{(r)}$ and
\[
      \forall\ell\in\Lambda,\quad
n_\ell\coloneqq (\C^{(r)}:H_\ell)=\#
      C_\ell=c^{\binom{r+j}{r}}.
\]
\end{lemma}

We may now define on~$\C^{(r)}$ a random walk~$(X_k)$
that satisfies the requirements of Theorem~\ref{thm:LS}. We then ask the
question:

\noindent
\emph{at which speed do we reach a colouring~$X_k$ of the $r$-element subsets
of~$X$ that exhibits a subset~$A\subseteq X$ of size~$r+j$, all the
$r$-element subsets of which have the same colour?}

\noindent The next statement answers that question.
\begin{theorem}\label{th:InfinteRamseyLike}
Let~$(X_k)$ be the random walk defined on~$\C^{(r)}$ as in
Subsection~\ref{sub:rw}, with
$\tilde{S}_\ell(b,\delta)\subseteq C_\ell$.  Fix positive integers~$j$, $r$, $c$
and a positive real number~$\eps$.
Then for every positive integer~$k$,
\[
     \Pe\left(\text{No element of\, $\bN^{(j+r)}$ has all its $r$-element
     subsets of the same colour in~$X_k$}\right)
\ll k^{-1/2+\eps},
\]
where the implied constant depends only on~$\eps$, $j$, $r$, $c$, $C_0$,
$S(b,\delta)$, and the sequence~$(p_s)$. This constant could be explicitly
computed in terms of~$\eps$, $j$, $r$, $c$, and the constant~$\nu$ of
Theorem~\ref{thm:LS}. As a function of~$j$, this constant is unbounded.
\end{theorem}
\begin{proof} According to Lemma~\ref{lem:coloringInfiniteRamsey}, we know that
      $n_\ell$ is independent of~$\ell$ since $n_\ell=c^{\binom{r+j}{r}}$.  Let us
      set~$b_\ell\coloneqq\ell$ for all~$\lambda\in\Lambda$.  Thus
      $\kappa(b_\ell,\ell;\delta)\geq\psi(\delta)(\log n_\ell+\ell+\log 2)$, where
      $\psi$ is defined by~\eqref{eq:psi}. Thus $n_\ell\leq
      \kappa(b_\ell,\ell;\delta)$ if $\ell$ is large enough, e.g.,~if $\ell\ge
      L_0\coloneqq c^{\binom{r+j}{j}}$. Moreover, if for each~$\ell\in\Lambda_L$
      the set~$S_\ell(b_\ell;\delta)$ contains at least $n_\ell/2$ distinct
      elements, then~$C_0\le2$.  Therefore, applying Lemma~\ref{lem:proba} we
      deduce that for for every~$L\geq L_0$, \begin{align*} \Pe(C_0=\infty)&\le
            \Pe(\exists \ell \in \Lambda_L, \, \# S_\ell(b_\ell;\delta)<
            n_\ell/2)\\ &\le
            \phi_0(c,r,j)\cdot\sum_{\ell\in\Lambda_L}2^{-\kappa(b_\ell,\ell;\delta)},
      \end{align*} where~$\phi_0(c,r,j)$ is a number depending only on~$c$, $r$
      and~$j$.  Next, as $\ell\le\kappa(b_\ell,\ell;\delta)$ we deduce that
\[
      \forall\ell\ge L_0,\quad \Pe(C_0=\infty)\le \phi_0(c,r,j)2^{-L+1}.
\]
Fix a positive integer~$k$ and a positive real number~$\eps$. 
For each~$\ell\in \Lambda$, we set
\[
   \Theta_\ell\coloneqq\{g\in\C^{(r)}/H_\ell\colon \text{the only representative
   of~$g$ in~$C_\ell$ is constant on~$E_\ell^{(r)}$}\}.
\]
Of course, $\#\Theta_\ell/n_\ell=c/n_\ell=c^{-\binom{r+j}{r}+1}$. Before going further, we note the existence of a
constant~$\psi_1(c,r,j)$ depending only on~$c$, $r$ and~$j$ such that 
$\kappa(n_\ell,b_\ell;\delta)\le 8\ell+\psi_1(c,r,j)$.
Lemma~\ref{lem:coloringInfiniteRamsey} ensures that the hypotheses
of Theorem~\ref{thm:LS} are satisfied. Thus, abbreviating~$\kappa_\ell(b_\ell,\ell;\delta)$ as~$\kappa_\ell$, we obtain
\begin{align*}
\Pe(\rho_\ell(X_k)\not\in\Theta_\ell,\,\forall \ell\in \Lambda_{L})-\Pe(C_0=\infty)
&\le \left(1+c^{\binom{r+j}{r}}\abs{\Lambda_L}(1-\nu\kappa_{2L}^{-2})^k\right)\left(\sum_{\ell=L}^{2L}\frac{\abs{\Theta_\ell}}{n_\ell}\right)^{-1}\\
&\quad\quad+\sum_{\ell=L}^{2L}e^{-b_\ell}\\
&\le\left(1+(L+1)\cdot c^{\binom{r+j}{r}}(1-\nu\kappa_{2L}^{-2})^k\right)\frac{c^{\binom{r+j}{r}}-1}{L}\\
&\quad\quad +e^{1-L}-e^{-2L}\\
&\le\frac{c^{\binom{r+j}{r}}}{L}+2c^{2\binom{r+j}{r}}(1-\nu(16L+\psi_1(c,r,j))^{-2})^k,
\end{align*}
where we use the inequality~$e^{1-x}-e^{-2x}\le1/x$ if $x>0$.
For any fixed~$\eps>0$, set~$L\coloneqq\lceil k^{1/2-\eps}\rceil$. For this to be compatible with
the condition~$L\ge L_0$, it is enough that $k^{1/2-2\eps}\ge c^{\binom{r+j}{r}}$. This inequality
can be made to hold by modifying the implied constant in the estimate proven, thereby making it depend
also on~$\eps$.
We compare the order of magnitude of 
the two summands of the above right side with the upper bound obtained for~$\Pe(C_0=\infty)$. First, one has
\[
\Pe(C_0=\infty)\ll_{L_0} k^{-1/2+\eps}.
\]
Also, since $L\geq 1$ we know that $16L+\psi_1(c,r,j)\le (16+\psi_1(c,r,j))L$, so
\begin{align*}
(1-(16L+\psi_1(c,r,j))^{-2}\nu)^k&=\exp\left(-\frac{\nu}{(16+\psi_1(c,r,j))^2}k^{2\eps}+\bigo{\nu k^{-1+4\eps}}\right)\\
&\ll \exp\left(-\frac{\nu}{(16+\psi_1(c,r,j))^2}k^{2\eps}\right),
\end{align*}
with an absolute implied constant. We thus obtainthe upper bound
\[
   \varphi(\eps,r,j,c,C_0,S(b,\delta),(p_s))k^{-1/2+\eps}
\]
for the probability investigated, where
$\varphi(\eps,r,j,c,C_0,S(b,\delta),(p_s))$ is a positive constant depending only on the tuple~$(\eps,r,j,c,C_0,S(b,\delta),(p_s))$. This finishes the proof.
\end{proof}

\begin{remark}
We point out an important limitation to
our approach: we cannot dispense of the use of the auxiliary parameter~$j$.  More
precisely, letting~$j$ tend to infinity in the inequality of
Theorem~\ref{th:InfinteRamseyLike} yields only a trivial upper bound for the
probability investigated (this comes from the unboundedness of the implied
constant in Theorem~\ref{th:InfinteRamseyLike} as a function of~$j$).
\end{remark}

\subsection{Monochromatic Solutions to Equations}\label{sub:sol}
It also seems relevant to study solutions of an equation through the
perspective of Ramsey Theory: can one destroy the solutions of an equation by
partitioning the different values the variables can take? We are interested in the
following question, which turns out to fit our setting.

\noindent
\emph{Given a random $c$-colouring of a random subset~$A$ of~$\mathbf{Z}$, what is
the probability that $A$ contains a monochromatic non-empty subset summing to~$0$?}

\noindent
We study this question in two steps. First we leave
aside colourings and just bound the probability that a random subset of~$\bZ\setminus\{0\}$
contains no subset summing to~$0$. To this end, the group~$G$ considered is
that of all subsets of~$\bZ\setminus\{0\}$ with the symmetric difference $\Delta$ as group
law.  We then show how easily one can add constraints on colourings to this setting, by just
considering the product of the group~$G$ with the group of all $c$-colourings
of~$\bZ\setminus\{0\}$. So in our setting the coloured version essentially reduces
to the first question.

Let~$G$ be the group consisting of all subsets of~$\bZ\setminus\{0\}$ endowed
with the symmetric difference.  For each positive integer~$\ell$ (i.e.,~we choose
$\Lambda\coloneqq\bN_{\ge 1}$), we set~$I_\ell\coloneqq\{-\ell,\ell\}$,
$C_\ell\coloneqq2^{I_\ell}$ and we define~$H_\ell$ to be the subgroup of~$G$
consisting of all subsets of~$\bZ\setminus\{0\}$ that are disjoint from~$I_\ell$.
Thus $C_\ell$ forms a set of representatives for~$G_\ell\coloneqq G/H_\ell$.  In
particular, $n_\ell\coloneqq[G:H_\ell]=4$ and $(H_\ell,C_\ell)_{\ell\ge1}$ is an
admissible local sequence for~$G$ (since the sets~$I_\ell$ are pairwise disjoint).

\noindent
We set
\[\Theta_\ell\coloneqq\{X\in G_\ell\colon \forall \tilde{X}\in G,\quad\rho_\ell(\tilde{X})=X \Rightarrow
\sum_{x\in \tilde{X}}x=0\},
\]
so $\Theta_\ell$ is a singleton, the unique element of which is represented by~$I_\ell$.
Now one can define a random walk~$(X_k)$ as in Subsection~\ref{sub:rw}. This
random walk readily satisfies the requirements of Theorem~\ref{thm:LS}.
Further, observe that if a subset~$S$ of~$\bZ$ does not contain a non-empty
subset summing to~$0$, then neither does the intersection of~$S$ with any
fixed subset. Thus the probability~$P_k$ that $X_k$ does not contain a
non-empty subset summing to~$0$ is at most
\[
\Pe(\rho_\ell(X_k)\notin\Theta_\ell,\,\forall\ell).
\]
We choose $b_\ell=\ell$ for all~$\ell\in \Lambda$ and apply the same method as in the proof of 
Theorem~\ref{th:InfinteRamseyLike} to bound~$\sum_{\ell\in\Lambda_L}b_\ell$ from above.
Since $\abs{\Theta_\ell}/n_\ell=\frac{1}{4}$ for each positive integer~$\ell$,
Theorem~\ref{thm:LS} implies that for every positive
integer~$L$ and every positive integer~$k$,
\[
   P_k - \Pe(C_0=\infty)\le\frac{1}{L}+\left(1+(L+1)\max_{L\le\ell\le2L}\abs{G_\ell}(1-\nu\kappa_{2L}^{-2})^k\right)\cdot\frac{4}{L}
       =\frac{5}{L}+8(1-\nu\kappa_{2L}^{-2})^k.
\]
We have $\kappa(b_\ell,\ell;\delta)\ge \ell$ thus $\kappa(b_\ell,\ell;\delta)\ge n_\ell$ whenever~$\ell\ge 4$. Applying 
Lemma~\ref{lem:lindisjoint} we obtain in the same way as before
\[
\Pe(C_0=\infty)\leq 6\sum_{\ell=L}^{2L}2^{-\ell},
\]
for any choice of~$L\geq 4$. Observing that $\kappa_{2L}\le 32 L$ by~\eqref{eq:boundpsi}, and setting~$L\coloneqq \lceil k^{1/2-\eps}\rceil$ we compute as 
in the proof of Theorem~\ref{th:InfinteRamseyLike},
\[
(1-\nu\kappa_{2L}^{-2})^k\ll_\nu \exp\left(-\frac{\nu k^{2\eps}}{32^2}\right).
\]

Consequently we infer the following statement
\begin{theorem}\label{th:zerosum}
Let~$(X_k)$ be a random walk on~$G$ defined as in Subsection~\ref{sub:rw}
using~$S(b,\delta)$, with~$\tilde{S}_\ell(b,\delta)\subseteq2^{I_\ell}$. 
Then for all~$\eps>0$ there exists 
a positive constant~$C_\eps$ (that can be computed explicitly as a function of~$\eps$ and $\nu$) depending 
only on~$\eps$, $S(b,\delta)$, $C_0$, and the sequence~$(p_s)$, such that for every positive integer~$k$
\[
   \Pe(\text{$X_k$ does not contain a non-empty subset summing to
   $0$})\leq C_\eps k^{-1/2+\eps}.
\]
\end{theorem}

Let us now see how to deal with the coloured version, that is, we want to
upper bound the probability that our random $c$-coloured subset does not
contain a \emph{monochromatic} non-empty subset summing to~$0$, where~$c$ is
an integer greater than~$1$.  It suffices to work in the product group
$G\coloneqq (2^{\bZ\setminus\{0\}},\Delta)\times\{f\colon\bZ\setminus\{0\}\to\bZ/c\bZ\}$. For each
positive integer~$\ell$ (i.e.,~we choose $\Lambda\coloneqq\bN_{\ge1}$), the subgroup~$H_\ell$ is defined to be
\[
      2^{\bZ\setminus
(I_\ell\cup \{0\})}\times\{f\colon\bZ\setminus\{0\}\to\bZ/c\bZ\colon f(-\ell)=f(\ell)=0\}
\]
where~$I_\ell$ is~$\{-\ell,\ell\}$ as before.

Thus $n_\ell\coloneqq[G:H_\ell]=4\cdot2^c=2^{c+2}$, which does not depend on~$\ell$.
A set of representatives for~$G_\ell\coloneqq G/H_\ell$ is
   $2^{I_\ell}\times \mathscr{F}_\ell$
where \[
\mathscr{F}_\ell\coloneqq
\{f\colon\bZ\setminus\{0\}\to\bZ/c\bZ\,:\, f\rest(\bZ\setminus (I_\ell\cup\{0\}))=0\}.
\]
Again since the sets~$I_\ell$ are pairwise disjoint 
the sequence~$(H_\ell,\mathscr{F}_\ell)$ is an admissible local sequence for~$G$.

Defining~$\Theta_\ell$ to be~$\{I_\ell\}\times
\{f\colon\bZ\setminus\{0\}\to\bZ/c\bZ\,:\,\text{$f$ is constant}\}$, it follows that
$\abs{\Theta_\ell}/n_\ell=c2^{-c-2}$. Since the hypotheses of
Theorem~\ref{thm:LS} are satisfied, one obtains the following statement.
(The proof goes along the same lines as that of Theorem~\ref{th:zerosum} --- in particular we choose~$b_\ell=\ell$ and, for any fixed~$\eps>0$, we set~$L\coloneqq\lceil k^{1/2-\eps}\rceil$. Details are omitted.)
\begin{theorem}\label{th:ColZeroSum}
Let~$(X_k)$ be a random walk on~$G$ defined as in Subsection~\ref{sub:rw}
using~$S(b,\delta)$, with
$\tilde{S}_\ell(b,\delta)\subseteq2^{I_\ell}\times\mathscr{F}_\ell$.
Then for every positive real number~$\eps$ and every positive integer~$k$, one has 
\[
   \Pe(\text{$X_k$ does not contain a monochromatic non-empty subset summing to
   $0$})\ll_\eps k^{-1/2+\eps},
\]
where the implied constant could be explicitly computed as a function of~$\eps$, $c$ and~$\nu$ (with notation as in Theorem~\ref{thm:LS}) and depends only on~$\eps$, $c$, $S(b,\delta)$, $C_0$ and the sequence~$(p_s)$.
\end{theorem}

At this point, it seems relevant to also discuss Ramsey theory for graphs.

\subsection{Looking for Monochromatic Triangles}\label{sub:monotri}
We let~$\mathscr{G}$ be the (countable) infinite complete graph, that is, the
graph with vertex set~$\bN$ in which every two distinct positive integers are
neighbours.  We fix an integer~$c\ge3$ and we define~$\C$ to be the collection
of all functions from the edges of~$\mathscr{G}$ to $\bZ/c\bZ$.  
As earlier, the set~$\C$ is naturally endowed with a group structure inherited from
that of~$\bZ/c\bZ$.

We are interested in monochromatic substructures of a given fixed size that
may arise.  Specifically, to avoid unnecessary notation and abstraction, we
shall focus on finding monochromatic triangles --- though our strategy could be
adapted effortlessly to the question of detecting monochromatic $r$-cliques or
$r$-cycles for~$r\geq 3$.

We define a family of subgroups~$(H_\ell)_{\ell\in\Lambda}$ of~$\C$, where
$\Lambda\coloneqq\bN$. Consider a partition \emph{in finite parts}
$(I_\ell)_{\ell\in\Lambda}$ of~$\Lambda$. We set~$i(\ell)\coloneqq\abs{I_\ell}$
for~$\ell\in\Lambda$.  Let $E_\ell\coloneqq\{(a,b)\in I_\ell^2\colon a\neq b\}$,
that is, $E_{\ell}$ is the set of all edges of~$\mathscr{G}$ with both endvertices
contained in~$I_\ell$.  We define~$C_\ell$ to be the collection of all
functions~$f\in\C$ with support contained in~$E_\ell$.  Then $H_\ell$ is the
collection of all functions~$f\in\C$ such that $f\rest{E_\ell}\equiv0$.

Let us give the necessary properties that the quotients~$\C_\ell\coloneqq \C/H_\ell$ satisfy.
\begin{lemma}\label{lem:coloring}
 For each $\ell\in\Lambda$, the following holds.
    \begin{enumerate}
         \item[(i)] $C_\ell$ is a set of representatives of the quotient $\C_\ell$ and $(H_\ell,\C_\ell)$ is an admissible local sequence for $\C$;
            and
         \item[(ii)] the index of~$H_\ell$ in~$\C$ is
            $n_\ell\coloneqq[\C:H_\ell]=\abs{C_\ell}=c^{i(\ell)(i(\ell)-1)/2}$.
\end{enumerate}
\end{lemma}
\begin{proof}
(i) No two distinct functions in~$C_\ell$ are congruent modulo an
element of~$H_\ell$.  Moreover, for any $f\in\C$, let~$f_C$ be the function
equal to~$f$ on~$E_\ell$ and equal to~$0$ everywhere else, that is,
$f_C\rest{E_\ell}\coloneqq f\rest{E_\ell}$ and
$f_C\rest{(E(\mathscr{G})\setminus E_\ell)}\coloneqq0$. It follows that
$f_C\in C_\ell$ and $f-f_C\in H_\ell$, or equivalently $f\equiv f_C (\bmod
H_\ell)$. Finally the fact that $(H_\ell,C_\ell)$ is an admissible local sequence for 
$\C$ is a consequence of the disjointness of the sets~$I_\ell$.

(ii) By the definition, $\abs{E_\ell}=i(\ell)(i(\ell)-1)/2$.  The
conclusion follows.
\end{proof}

A practical way to rephrase part of the proof of Lemma~\ref{lem:coloring} is
to say that for each fixed integer~$\ell$ in~$\Lambda$ and each element~$f$ of~$\C$,
the unique element in~$C_\ell$ congruent to~$f$ modulo~$H_\ell$ is the
function equal to~$f$ on~$E_\ell$ and to~$0$ outside of~$E_\ell$.

From now on, we assume that $i(\ell)\ge3$ for~$\ell\in\Lambda$.  For each
integer~$\ell\in\Lambda$, let~$\Theta_\ell$ be the set of classes
$\bar{f}\in\C_\ell$ such that the unique representative~$f$ of~$\bar{f}$
in~$C_\ell$ (the existence of which is asserted by Lemma~\ref{lem:coloring})
contains a monochromatic triangle in~$E_\ell$. In other words
$f\in\Theta_\ell$ if and only if $I_\ell$ contains three integers~$i_1$, $i_2$
and~$i_3$ such that $f((i_1,i_2))=f((i_1,i_3))=f((i_2,i_3))$. Observe that
$\abs{\Theta_\ell}/\abs{\C_\ell}\ge c^{-2}$. Indeed any function that
restricts to a constant map (with values in~$\bZ/c\bZ$) on a fixed triangle
contained in~$E_\ell$ surjects to an element of~$\Theta_\ell$ \emph{via}~$\rho_\ell$.

Assume that $\delta$ is a fixed real number in~$(0,1/2]$. We set~$b_\ell\coloneqq\ell$. In particular, note that
\[
\kappa(b_\ell,\ell;\delta)=\left\lceil\psi(\delta)\cdot\left(\frac{i(\ell)(i(\ell)-1)\log
c}{2}+\ell+\log2\right)\right\rceil.
\]

Given~$f^{(\ell)}\in S_\ell(b_\ell,\delta)$, we define~$\tilde{f}^{(\ell)}$ to
be its canonical representative in~$\C$, that is, $\tilde{f}^{(\ell)}\in
C_\ell$. In this context, the outcome of Theorem~\ref{thm:LS} is the following.
\begin{proposition}\label{prop:coloring}
Let~$(X_k)$ be a random walk on~$\C$ defined as in
Subsection~\ref{sub:rw} using~$S(b,\delta)$ (see~\eqref{def:S}) with~$\tilde{S}_\ell(b,\delta)\subseteq C_\ell$.
Then with notation as in Theorem~\ref{thm:LS}, one has for any fixed positive integers~$L$ and~$k$
\begin{align*}
   \Pe(\text{$X_k$ does not contain a monochromatic triangle})&\le \Pe(C_0=\infty)+   \frac{c^2+1}{L}\\
&+2c^{(1/2)\cdot i(2L)(i(2L)-1)+2}(1-\kappa_{2L}^{-2}\nu)^k.
\end{align*}
\end{proposition}
\begin{proof}
Fix a positive integer~$k$. Lemma~\ref{lem:coloring} ensures that the
hypotheses of Theorem~\ref{thm:LS} are satisfied.  We obtain, applying
Theorem~\ref{thm:LS},
\begin{align*}
   \Pe(\rho_\ell(X_k)\not\in\Theta_\ell,\,\forall \ell\in \Lambda_{L})&-\!\!\Pe(C_0=\infty)
   \le\left(1+\left(\sum_{\ell\in\Lambda_L}n_\ell\right)\!\!(1-\kappa_{2L}^{-2}\nu)^k\right)\left(\sum_{\ell\in\Lambda_L}\frac{\abs{\Theta_\ell}}{n_\ell}\right)^{\!\!-1}\\&+\sum_{\ell=L}^{2L}e^{-b_\ell}\\
  &\le e^{1-L}-e^{-2L}+\left(1+(L+1)\cdot c^{i(2L)(i(2L)-1)/2}(1-\kappa_{2L}^{-2}\nu)^k\right)\cdot\frac{c^2}{L}\\
   &\le\frac{c^2+1}{L}+2c^{(1/2)\cdot i(2L)(i(2L)-1)+2}(1-\kappa_{2L}^{-2}\nu)^k,\\
\end{align*}
where we used that $e^{1-x}-e^{-2x}\le 1/x$ for~$x\ge1$.
\end{proof}

Different choices of sets~$I_\ell$ may correspond to different speeds of
rarefaction of non-typical structures. (We note, however, that the random walk itself does depend
on the choice made for the sets~$I_\ell$.) More precisely, one can put additional
constraints on the structure of the monochromatic triangles, e.g.,~we 
may impose the three vertices to be consecutive integers as in the following
theorem.
\begin{theorem}\label{th:col1}
With notation as in Proposition~\ref{prop:coloring}, one has for every positive real number~$\eps$ and 
for every positive integer~$k$,
\begin{align*}
   & \Pe(\text{$X_k$ does not contain a monochromatic
  triangle})\\
  \le &\Pe(\text{$X_k$ does not contain a monochromatic
  triangle on three consecutive vertices})\\
  \ll_\eps & k^{-1/2+\eps},
\end{align*}
where the implied constant 
can computed explicitly (as a function of~$\eps$, $c$ and~$\nu$ (see Theorem~\ref{thm:LS})), and depends only on~$\eps$, $c$, $S(b,\delta)$, $C_0$,  and the sequence~$(p_s)$.
\end{theorem}
\begin{proof}
Set~$I_\ell\coloneqq\{3\ell-2,3\ell-1,3\ell\}$ for each~$\ell\in\Lambda$. In
particular $i(\ell)(i(\ell)-1)=6$. 
 Let us evaluate~$\Pe(C_0=\infty)$. One has
    $\#S_\ell(b_\ell;\delta)\le n_\ell=c^3$ and $\kappa(b_\ell,\ell;\delta)=
    \lceil\psi(\delta)(3\log c+\ell+\log 2)\rceil\geq \ell$, by~\eqref{eq:boundpsi}.
     In particular $n_\ell\leq \kappa(b_\ell,\ell;\delta)$ for all $\ell\geq c^3$. Moreover if we assume that for all
     $\ell\in\Lambda_L$ the set $S_\ell(b_\ell;\delta)$ contains at least $n_\ell/2=c^3/2$ distinct elements then $C_0\le 2$. Therefore
    \begin{align*}
    \Pe(C_0=\infty)&\le \Pe(\exists \ell \in \Lambda_L, \, \# S_\ell(b_\ell;\delta)< n_\ell/2)\\
    &\le \binom{c^3}{\lceil \frac{c^3}{2}\rceil}\sum_{\ell\in\Lambda_L}2^{-\kappa(b_\ell,\ell;\delta)},
    \end{align*}
    for all~$L\ge c^3$, by virtue of Lemma~\ref{lem:proba}.

 For~$\eps>0$ fixed, set~$L\coloneqq \lceil k^{1/2-\eps\rceil}$. For this to be compatible with the condition~$L\geq c^3$ 
 we need to have $k^{1-2\eps}\geq c^6$. This inequality can be made to hold by modifying the implied constant 
 in the estimate to be proven.
  As in the proof of Theorem~\ref{th:InfinteRamseyLike} we have
    \[
    \Pe(C_0=\infty)\ll_c k^{-1/2+\eps}.
    \]
    Moreover Proposition~\ref{prop:coloring} implies that
\begin{align*}
 &  \Pe(\text{$X_k$ does not contain a monochromatic
   triangle on three consecutive vertices})\\
   &\le\Pe(C_0=\infty)+\frac{c^2+1}{k^{1/2-\eps}}+2c^{5}(1-\nu\kappa_{2L}^{-2})^k.
\end{align*}
To find an upper bound for the third summand we first use the assumption~$L\geq c^3$ to deduce $\kappa_{2L}\le 46L$ and then 
we compute
\[
(1-\nu\kappa_{2L}^{-2})^k=\exp\left(k\left(-\frac{\nu}{46^2k^{1-2\eps}}\right)+O(\nu k^{-2+4\eps})\right)\ll_{c,\nu}\exp\left(-\frac{\nu}{46^2}k^{2\eps}\right),
\]
which finishes the proof.
\end{proof}

We note that the contributions from the non-standard case (that is,
$X(G/H_{\ell},S_\ell(b_\ell;\delta))$ is not an expander) is the probability with highest order of magnitude (among the three 
summands in the upper bound of Theorem~\ref{thm:LS})
given our choice of parameters in the proof of the theorem.  It is natural to compare Theorem~\ref{th:col1} with what is known from
Ramsey theory; this discussion is deferred to the next section.

\section{Remarks and Further Applications}\label{section:remarks}
As mentioned earlier, the main purpose of our work is to obtain a general sieve
statement in a purely combinatorial setting. Regarding the illustrative
applications, the general line of thought is to give, for the intricate notion of
randomness defined, explicit upper bounds for probabilities that we expect to be
small.

Let us underline some peculiarities of the application developed in
Subsections~\ref{sub:sol} and~\ref{sub:monotri}. For monochromatic substructures, it follows from Ramsey's
theorem~\cite{Ram30} that for every fixed positive integer~$c$, there exists
an integer~$N$ such that if $n\ge N$, then every $c$-colouring of the edges of
the complete graph~$K_n$ on~$n$ vertices contains a monochromatic triangle.
Alon and R\"odl~\cite{AlRo05} established that the smallest such integer~$N$ is
$\Theta(3^{c})$ as $n$ tends to infinity (that is, there exist two constants~$\rho$ and~$\rho'$ such that for sufficiently large~$n$,
this value belongs to~$[\rho\cdot 3^c,\rho'\cdot3^c]$). In our setting, although the infinite
complete graph is involved, only finite subgraphs of it are checked for the
existence of monochromatic triangles. These subgraphs are not necessarily
large enough for Ramsey's theorem to apply. In addition, we only consider
monochromatic triangles with vertices contained in some prescribed set~$I_{\ell}$.

Another feature of the applications presented is
uniformity of the decay rate with respect to the number~$c$ of colors involved. Actually,
we even claim control of the dependency of the implied constant as a function of~$c$,
since this implied constant could be explicitly computed. No such
uniformity holds in the context of Ramsey theory. Indeed, as already
mentioned, Alon and R\"odl's theorem~\cite{AlRo05} asserts that the number of
required vertices for Ramsey's theorem to hold grows exponentially fast with~$c$.

Next let us comment on the common decay rate, roughly~$1/\sqrt{k}$, in our various applications.
When applying Theorem~\ref{thm:LS}, we always have to find an upper bound of the rough form
\[
 \sum_{L\leq\ell\leq 2L}2^{-\ell}+c_1\sum_{L\leq\ell\leq 2L} 2^{-c_2\ell}+\left(1+c_3L\left(1-\frac{c_4\nu}{L^2}\right)^k\right)\frac{c_5}{L},
\]
where each parameter~$c_i$ is an absolute constant.
 
The fast decay of the first two summands is not an issue as soon as~$L$ is chosen
to be roughly equal to some power of~$k$.  However in the third summand one has to
have simultaneously~$L\rightarrow \infty$ and $(1-c_4\nu L^{-2})^k\rightarrow 0$,
as $k\rightarrow \infty$. These constraints justify the choice~$L=\lceil
k^{1/2-\eps}\rceil$ in all our applications. There is certainly room for
improvement here (e.g. by choosing a different value for~$b_\ell$, rather than
setting~$b_\ell$ to be~$\ell$, or by modifying the sieve itself so that $n_\ell$ is not
necessarily bounded as a function of~$\ell$), but we feel that ensuring the decay
of the third summand will remain a rather serious constraint in general.

\par\medskip
We highlight a strategy similar to that used in Subsection~\ref{sub:sol}
that allows one to check for monochromatic arithmetic progressions for which the
length, the common difference and the ``shape'', are prescribed.  Fix positive
integers~$s$ (the desired length of the arithmetic progression), $q$ (the
desired common difference), and $c\ge3$ (the number of colours).  Similarly as
before, let~$\C$ be the group of all $c$-colourings of~$\bN$. We consider the
subsets~$I_\ell\coloneqq\{\ell sq, \ell sq+q,\dotsc, \ell sq+(s-1)q\}$
for~$\ell\in\Lambda\coloneqq\bN$.  (It is this choice of particular subsets
of~$\bN$ of length at least~$s$ that provides a control on the ``shape'' of the
arithmetic progressions to be found.) In this setting our method yields the
following result.
\begin{theorem}\label{th.vdwlike}
    Let~$(X_k)$ be a random walk on $\C$ defined as in Subsection~\ref{sub:rw}
    using $S(b,\delta)$ \emph{via} the admissible local
    sequence~$(H_\ell,C_\ell)$. For every positive real number~$\eps$ and every positive
    integer~$k$,
\begin{align*} \Pe(&\text{$X_k$ contains no monochromatic
          arithmetic progression}\\ &\text{ with common difference~$q$ and
          length~$s$})\ll k^{-1/2+\eps},
\end{align*}
      where the implied constant could be computed explicitly as a function
      of~$(c,s,q,\nu)$ (see Theorem~\ref{thm:LS} for the definition of~$\nu$) and
      depends only on~$(c,s,q)$, and on~$C_0$, $S(b,\delta)$, and the sequence~$(p_s)$.
\end{theorem}

Let us sketch briefly the proof. For each~$\ell\in\bN$, let~$H_{\ell}$ be the set of
all functions~$f\colon\bN\to[c]$ such that $f\rest{I_{\ell}}\equiv0$.  The
index in~$\C$ of each of these subgroups is~$c^{s}$.  Moreover, there is
a collection of natural representatives~$C_\ell$ for the classes
modulo~$H_{\ell}$, namely the functions with support contained in~$I_{\ell}$. Thus
$n_\ell=c^{\# I_\ell}$ is independent of~$\ell$, and since the intervals~$C_\ell$
are pairwise disjoint, the sequence~$(H_\ell,C_\ell)$ is an admissible local
sequence for~$\mathscr{C}$. Let~$\Theta_{\ell}$ be the set of classes
modulo~$H_{\ell}$ whose unique representative in~$C_\ell$ contains a monochromatic
arithmetic progression of length~$s$ that is contained in~$I_{\ell}$. Then one has
$\abs{\Theta_{\ell}}/n_{\ell}\ge c^{-s}$.

Again we may apply
Theorem~\ref{thm:LS} with $b_\ell=\ell$ for all~$\ell\in\Lambda$.  Similarly as
before, $\Pe(C_0=\infty)$ can be bounded from above: if $L\ge c^{(s-1)q+1}$, then
$\kappa(b_{\ell},\ell,\delta)\ge n_{\ell}$ whenever~$\ell\ge L$, hence
Lemma~\ref{lem:proba}\ref{it:lemprobaii} yields that
$\Pe(C_0=\infty)\le\binom{c^{s}}{\lceil
      c^{s}/2\rceil}\sum_{\ell\in\Lambda_L}2^{-\ell}$. Therefore,
$\Pe(C_0=\infty)\le\binom{c^{s}}{\lceil c^{s}/2\rceil}/L$.

By Theorem~\ref{thm:LS}, the
probability that in~$X_k$ no monochromatic arithmetic progression with common
difference~$q$ and length~$s$ is contained in~$I_{\ell}$, for all~$\ell$ in~$\Lambda_{L}$ is at most
\[
\frac{\binom{c^{s}}{\lceil c^{s}/2\rceil}}{L}+\frac{1}{L}+\left(1+(L+1)c^{(s-1)q+1}(1-\nu\kappa_{2L}^{-2})^k\right)(c^sL)^{-1}.
\]
Since this last probability is, for every~$L$, an upper bound on the
probability that there is no monochromatic arithmetic progression in~$X_k$
with common difference~$q$ and length~$s$, Theorem~\ref{th.vdwlike} follows by
setting for any fixed~$\eps>0$ and~$L\coloneqq \lceil k^{-1/2+\eps}\rceil$.

\par\medskip
We conclude by pointing out the following: van der Waerden's
theorem~\cite{vdW27} ensures that, for each fixed positive integer~$s$ and each integer~$c\ge3$,
there exists an integer~$N$ such that if $n\ge N$
then any $c$-colouring of~$[n]$ yields a monochromatic arithmetic progression
of length~$s$. In the above setting, we impose two additional conditions: the
common difference of the arithmetic progression and a constraint on its form
(it must be contained in one of the sets~$I_\ell$). Van der Waerden's
theorem does not guarantee the existence of such an arithmetic progression
and the aforementioned inequality is essentially an explicit lower bound on the
speed of rarefaction of the colourings that do not yield a monochromatic
arithmetic progression with the required properties. Furthermore, and as mentioned
in the remarks about Subsections~\ref{sub:sol} and~\ref{sub:monotri}, the
uniformity of the decay rate with respect to the number of colours~$c$ is a quite
interesting by-product of our approach.

\end{document}